\documentclass[a4paper]{amsart} % sets paper format is a4
\usepackage[ascii]{inputenc} % for compatibility, please do not use utf8, let alone latin1
\usepackage{amssymb}
\usepackage{mathrsfs}

\usepackage{xcolor}

\usepackage[shortlabels]{enumitem}
\setlist[enumerate,1]{label={(\Alph*)}}
\setlist[enumerate,2]{label={(\alph*)}}
\setlist[enumerate,3]{label={$\bullet_{\arabic*}$}}

\newtheorem{theorem}{Theorem}[section]
\newtheorem{maintheorem}{Main Theorem}[section]

\newtheorem{claim}[theorem]{Claim}

\newtheorem{corollary}[theorem]{Corollary}

\newtheorem{lemma}[theorem]{Lemma}

\theoremstyle{definition}

\newtheorem{definition}[theorem]{Definition}

\theoremstyle{remark}

\newtheorem{question}[theorem]{Question}

\newcommand{\seq}[1]{{\langle{#1}\rangle}}
\newcommand{\restr}{\mathord{\upharpoonright}}
\renewcommand{\subset}{\subseteq}
\renewcommand{\epsilon}{\varepsilon}

\newcommand{\CH}{\mathrm{CH}}

\newcommand{\dom}{\mathrm{dom}}

\newcommand{\GCH}{\mathrm{GCH}}

\newcommand{\ZFC}{\enusuremath{\mathrm{ZFC}}}

\def\forces{\Vdash}

\def\ZFC{\mathsf{ZFC}}

\def\MM{\mathsf{MM}}
\def\PFA{\mathsf{PFA}}
\def\MA{\mathsf{MA}}

\def\BA{\mathsf{BA}}

\def\GCH {\mathsf{GCH}}
\def\CH {\mathsf{CH}}
\def\Q{\mathbb Q}
\def\P{\mathbb P}

\newcommand{\R}{\mathbb{R}}
\newcommand{\w}{\omega}

\newcommand{\ran}{\mathrm{ran}}

\def\cc{2^{\aleph_0}}

\newcount\skewfactor
\def\mathunderaccent#1#2 {\let\theaccent#1\skewfactor#2
\mathpalette\putaccentunder}
\def\putaccentunder#1#2{\oalign{$#1#2$\crcr\hidewidth
\vbox to.2ex{\hbox{$#1\skew\skewfactor\theaccent{}$}\vss}\hidewidth}}

\newbox\noforkbox \newdimen\forklinewidth
\setbox0\hbox{$\textstyle\bigcup$}
\setbox1\hbox to \wd0{\hfil\vrule width \forklinewidth depth \dp0
                        height \ht0 \hfil}
\setbox\noforkbox\hbox{\box1\box0\relax}
\def\unionstick{\mathop{\copy\noforkbox}\limits}
\def\nonfork#1#2_#3{#1\unionstick_{\textstyle #3}#2}
\def\nonforkin#1#2_#3^#4{#1\unionstick_{\textstyle #3}^{\textstyle
    #4}#2}
\setbox0\hbox{$\textstyle\bigcup$}
\setbox1\hbox to \wd0{\hfil{\sl /\/}\hfil}
\setbox2\hbox to \wd0{\hfil\vrule height \ht0 depth \dp0 width
                                \forklinewidth\hfil}
\newbox\doesforkbox
\setbox\doesforkbox\hbox{\box1\box0\relax}
\def\nunionstick{\mathop{\copy\doesforkbox}\limits}

\def\fork#1#2_#3{#1\nunionstick_{\textstyle #3}#2}
\def\forkin#1#2_#3^#4{#1\nunionstick_{\textstyle #3}^{\textstyle
    #4}#2}

\newcommand{\stickT}{%
\setbox255=\hbox{\raise1ex\hbox{$\hspace{0.2pt}\,\bullet\,$}}
\mathord{\rlap{\hbox to\wd255{\hss\hbox{$|$}\hss}}
\box255}
}
\newcommand{\stickS}{%
\setbox255=\hbox{\raise0.6ex\hbox{$\scriptstyle\bullet$}}
\mathord{\rlap{\hbox to\wd255{\hss\hbox{$\scriptstyle|$}\hss}}
\box255}
}

\usepackage{hyperref}

\author[P. Marun]{Pedro Marun}
\address[P. Marun]{
Institute of Mathematics, 
Czech Academy of Sciences, 
{\v Z}itn{\'a} 25, Prague 1, 
115 67, Czech Republic
}
\urladdr{https://pedromarun.github.io}
\author[S. Shelah]{Saharon Shelah}
\address[S. Shelah]{Einstein Institute of Mathematics,
The Hebrew University of Jerusalem,
9190401, Jerusalem, Israel; and\\
Department of Mathematics,
Rutgers University,
Piscataway, NJ 08854-8019, USA}
\urladdr{https://shelah.logic.at/}
\author[C. B. Switzer]{Corey Bacal Switzer}
\address[C. B. Switzer]{Kurt G\"{o}del Research Center, Faculty of Mathematics, University of Vienna, Kolingasse 14 -- 16, 1090 Wien, Austria}
\urladdr{}
\thanks{First typed 2025-12-19.
Research of the second author partially supported by the Israel Science Foundation (ISF) grant no: 2320/23. The first author was supported by the Czech Academy of Sciences (RVO 67985840). The research of the third author was funded in whole or in part by the Austrian Science Fund (FWF) grant doi 10.55776/ESP548. For open access purposes, the author has applied a CC BY public copyright license to any author-accepted manuscript version arising from this submission. References like [Sh:950, Th0.2=Ly5] mean that the internal label of Th0.2 is y5 in Sh:950.
The reader should note that the version on S. Shelah's website is usually more up-to-date than the one in arXiv.
This is publication number    
1272
in Saharon Shelah's list.
}

\makeatletter
\@namedef{subjclassname@2020}{\textup{2020} Mathematics Subject Classification}
\makeatother
\subjclass[2020]{03E05, 03E35, 03E50}
\keywords{$\aleph_1$-dense sets, Forcing, Baumgartner's Axiom, Martin's Axiom}
\date{December 24th, 2025} % add date of last edit here, do not use \today

\usepackage{changes}
\reversemarginpar

\definechangesauthor[name={Pedro Marun}, color=teal]{PM}
\definechangesauthor[name={Corey Switzer}, color=orange]{CS}

\title{Baumgartner's Axiom and Small Posets}

\begin{document}
\begin{abstract}
    We contribute to the study of $\aleph_1$-dense sets of reals, a mainstay in set theoretic research since Baumgartner's seminal work in the 70s. In particular, we show that it is consistent with $\MA$ that there exists an $\aleph_1$-dense set of reals $A$ so that, in any cardinal-preserving generic extension by a forcing of size $\aleph_1$, $A$ and $A^*$ do not contain uncountable subsets which are order isomorphic. This strengthens a result of Avraham and the second author and yields a different proof of a theorem of Moore and Todorcevic.
    %Strengthening a result of Avraham and the second author we prove the consistency of $\MA$ with the existence of an $\aleph_1$-dense set of reals $A \subseteq \mathbb R$ so that any forcing notion of size $\aleph_1$ which generically adds an isomorphism from $A$ to its reverse ordering must necessarily collapse $\omega_1$. In particular it is consistent with $\ZFC + \MA$ that no proper poset of size ${<}2^{\aleph_0}$ can add isomorphisms between prescribed $\aleph_1$-dense sets of reals.
\end{abstract}
\maketitle

\setcounter{section}{-1}
\section{Introduction}
%Hello. Test-citations: \cite{Sh:a} \cite{Sh:19} \cite{Sh:377}
%\begin{enumerate}
%    \item Level 1
%    \begin{enumerate}
%        \item Level 2
%       \begin{enumerate}
%            \item Level3
%        \end{enumerate}
%    \end{enumerate}
%\end{enumerate}

One of the most important early results in forcing theory and its applications is Baumgartner's theorem from \cite{Baum73} that consistently all $\aleph_1$-dense sets of reals are order isomorphic. This statement is now commonly known in the literature as {\em Baumgartner's Axiom} and denoted $\BA$, see \cite{weakBA} for more on the background. Here a set of reals is \textit{$\aleph_1$-dense} if between any two points there are $\aleph_1$-many. Clearly each such is isomorphic to one which is moreover $\aleph_1$-dense in the reals - i.e. has intersection size $\aleph_1$ with every nonempty real open interval. The idea behind Baumgartner's proof is to assume $\CH$ and show that, given a pair of $\aleph_1$-dense sets of reals $A$ and $B$, there is a ccc forcing notion of size $\aleph_1$ which generically adds an isomorphism between them. From the modern perspective it is then clear that a careful bookkeeping suffices to provide a model of $\BA + \cc = \aleph_2$. Baumgartner also showed later in \cite{BaumPFA} that $\BA$ also follows from $\PFA$. 

A natural question is whether Baumgartner's result actually relies on $\CH$ i.e. given two $\aleph_1$-dense linear orders $A$ and $B$ is there necessarily a ccc forcing notion which generically adds an isomorphism between them. This was asked by Baumgartner in \cite{Baum73}. Similarly one can ask whether $\BA$ follows from $\MA$ already and not just $\PFA$. The answer to both these questions is ``no", due to a celebrated result of Avraham and the second author from \cite{Sh:106}.

\begin{theorem}[{\cite[Theorem 2]{Sh:106}}]
For any regular $\kappa > \aleph_1$ it is consistent that $\cc = \kappa$, $\MA$ holds and there is an $\aleph_1$-dense $A \subseteq \mathbb R$ which is not isomorphic to its reverse ordering. 
\end{theorem}

Note that were there a ccc forcing notion in this model adding an isomorphism between $A$ and its reverse ordering then they would have to already be isomorphic by $\MA$. For this reason, in this model not only does $\BA$ fail but it can never be forced by a ccc forcing. Given this, it is natural to ask the following two questions.
\begin{question}
Given $A$ and $B$ which are $\aleph_1$-dense is there always a {\em proper} forcing notion of size less than the continuum which adds an isomorphism? %Is there even always a cardinal preserving forcing notion adding an isomorphism?
\end{question}

\begin{question}
Does $\PFA (\aleph_1$-sized posets) suffice to prove $\BA$? What about $\MM(\aleph_1$-sized posets)?
\end{question}
Note that Baumgartner's proof provides in $\ZFC$ a proper forcing of size $\cc$ which generically adds isomorphisms. In full generality these questions remain open however in this paper we provide the following partial answer.
\begin{maintheorem} \label{maintheorem1}
For any regular cardinal $\kappa > \aleph_1$ it is consistent that $\MA + \cc = \kappa$ holds and there is an $\aleph_1$-dense $A \subseteq \mathbb R$ so that any forcing notion of size $\aleph_1$ which generically adds an isomorphism between $A$ and its reverse ordering must necessarily collapse $\aleph_1$. 
\end{maintheorem}
In particular we have the following immediate corollary.

\begin{corollary}
It is consistent with $\ZFC$ that there are $A$ and $B$ which are $\aleph_1$-dense sets of reals so that no proper forcing of size less than continuum can add an isomorphism between them. 
\end{corollary}

We remark in relation to this that in \cite{MooreTod17} it was shown that forcing $\MA_{\aleph_2}$ in the natural way forces that there is an $\aleph_1$-dense linear order that cannot be made isomorphic to its reverse in any outer model of the universe with the same $\omega_2$, which provides a partial answer to the questions above as well. The results above allow us to give therefore an alternative proof of the observation made concerning the proof of \cite[Theorem 1.4]{MooreTod17} made in the introduction to that paper.

\begin{corollary}[{Moore-Todorcevic, See \cite[Theorem 1.4]{MooreTod17} and related discussions}]
$\MA + \neg \CH$ is consistent with an $\aleph_1$-dense linear order that cannot be made isomorphic to its reverse ordering by any $\aleph_1$-sized forcing notion which does not collapse $\omega_1$. 
\end{corollary}

Our proofs actually establish quite a bit more. The definition of a {\em good} linear order is given below, see Definition \ref{good definition}. For the purposes of the introduction we note that any set of uncountably many mutually generic Cohen reals is good. We show that if there are good, $\aleph_1$-dense linear orders, then under $\MA$ they necessarily have this indestructibility property. 

\begin{maintheorem} \label{mainthm2}
Assume $\MA_{\aleph_1}$. If $A \subseteq\mathbb R$ is good and $\aleph_1$-dense then it cannot be made isomorphic to its reverse order by any forcing notion of size $\aleph_1$ which does not collapse $\aleph_1$. 
\end{maintheorem}

Already in \cite{Sh:106} it was shown by the second author and Avraham that $\MA + 2^{\aleph_0} = \kappa$ is consistent with a good $\aleph_1$-dense linear order for any regular $\kappa$. This fact alongside Main Theorem \ref{mainthm2} implies Main Theorem \ref{maintheorem1}. 

The rest of this paper is organized as follows. In the next section we introduce some combinatorial properties of $\aleph_1$-dense linear orders which will be used in the proof of Main Theorem \ref{maintheorem1}. The following section proves that if $\MA$ holds alongside the existence of a linear order satisfying a certain combinatorial property, {\em slicewise coverable}, described in \S \ref{section: combinatorics}, then the conclusion of Main Theorem \ref{maintheorem1} holds. Finally in \S \ref{section: construction} we provide a model in which $\MA$ holds alongside such a linear order hence proving Main Theorem \ref{maintheorem1}. A final section concludes with some further observations and open questions. 

Throughout, our notation is standard, conforming to that of the monographs \cite{jechSetTheory2003} or \cite{kunenSetTheory2011}. We fix some notation for linear orders which is slightly less standard or well known. First, given a linear order $A \subseteq \mathbb R$ its {\rm reverse ordering} is denoted $A^* = \{-r: r\in A\}$. Such an order is {\em reversible} if $A \cong A^*$. Clearly $\BA$ implies all $\aleph_1$-dense linear orders are reversible. Next we say that two linear orders $A, B \subseteq \mathbb R$ are {\em near} if there are uncountable subsets $A' \subseteq A$ and $B' \subseteq B$ so that $A' \cong B'$. If $A$ and $B$ are not near then we say they are \emph{far}. Finally, let us say that a map $f:A \to B$ (for linear orders $A$ and $B$) is {\em increasing} if $a < b$ implies $f(a) < f(b)$. A map $f:A \to B$ is {\em decreasing} or {\em order reversing} if $a < b$ implies $f(b) < f(a)$. A map $f:A \to B$ is {\em monotone} if it is either increasing or decreasing.

\section{Anti-reversibility conditions} \label{section: combinatorics}
The key steps in the proof of Main Theorem \ref{maintheorem1} require the construction of a linear order which is very far from being reversible. In \cite{Sh:106} the following concept was used, though not given a name.

\begin{definition}
Let $A \subseteq \mathbb R$ be an $\aleph_1$-dense linear order. We say that $A$ is {\em essentially increasing} if every partial, uncountable injection $f:A \rightharpoonup A$ is increasing on some uncountable subset. 

\end{definition}
Obviously, if $A$ is essentially increasing then it is far from $A^*$. Moreover, the existence of an essentially increasing $A$ implies that no ccc forcing notion can force $A$ and $A^*$ to be near, see \cite[p. 106]{Sh:106}. In \cite{Sh:106} the following was proved.
\begin{theorem}[Avraham-Shelah, \cite{Sh:106}]
For any regular $\kappa > \aleph_1$ it is consistent that $\cc = \kappa$, $\MA$ holds and there is an essentially increasing linear order. 
\end{theorem}
We will need a strengthening of being essentially increasing, which we call being slicewise coverable. Below we say that a {\em filtration} of an $\aleph_1$-dense linear order $A$ is a $\subseteq$-increasing and continuous sequence $\vec{A} = \{A_\alpha: \alpha \in \omega_1\}$ so that the following are satisfied.
\begin{enumerate}
\item $A_0 = \emptyset$
\item For all $\alpha < \omega_1$ we have that $A_{\alpha+1} \setminus A_\alpha$ is a countable, dense subset of $A$.
\item $A = \bigcup_{\alpha < \w_1} A_\alpha$
\end{enumerate}
Note that if $\vec{A}$ is a filtration of $A$ then the set $\{A_{\alpha+1}\setminus A_\alpha: \alpha \in \omega_1\}$ forms a partition of $A$ into $\omega_1$-many countable, dense sets. In this set up let us refer to each $A_{\alpha+1}\setminus A_\alpha$ as a {\em slice} of the the filtration.

We now define slicewise coverability.

\begin{definition}
An $\aleph_1$-dense linear order is said to be {\em slicewise coverable} if given any filtration $\vec{A}$ of $A$ there are countably many partial increasing functions $\{f_n\}_{n < \omega}$ with $f_n : A \rightharpoonup A$ so that $\bigcup_{n < \omega} f_n = \bigcup_{\alpha < \omega_1} (A_{\alpha+1} \setminus A_\alpha) \times (A_{\alpha+1} \setminus A_\alpha)$.
\end{definition}
Note that the above says concretely that for every $n < \omega$ and every $\alpha \in \omega_1$ if $a \in {\rm dom}(f_n) \cap (A_{\alpha+1} \setminus A_\alpha)$ then $f_n(a) \in (A_{\alpha+1} \setminus A_\alpha)$ and, moreover every pair $(a, b) \in (A_{\alpha+1} \setminus A_\alpha)^2$ appears in this way for some $n < \omega$. In other words, each square of a slice is covered by the graphs of the $f_n$'s. 

\begin{lemma}
Suppose $A \subseteq \mathbb R$ is $\aleph_1$-dense. If $A$ is slicewise coverable, then it  is essentially increasing. 
\end{lemma}

\begin{proof}
Assume $A$ is a slicewise coverable, $\aleph_1$-dense linear order. Let $f:A \rightharpoonup A$ be a partial, uncountable function. Take an $\in$-increasing and continuous chain $\seq{M_\alpha:\alpha<\w_1}$ of countable elementary submodels of some large $H_\theta$ with $f,A\in M_0$. Let $A_\alpha:=M_\alpha\cap A$. This produces a filtration of $A$ so that each slice if closed under $f$, by elementarity. %Inductively define a filtration $\vec{A} = {A_\alpha:\alpha<\w_1}$ of $A$ as follows. First let $A_0 = \emptyset$. Since the filtration must be continuous we have that $A_\lambda = \bigcup_{\beta < \lambda} A_\beta$ for limits. Now assume $\{A_\alpha:\alpha<\beta\}$ for some $\beta < \omega_1$ has been defined and for each $\alpha +1 < \beta$ we have that $A_{\alpha+1}\setminus A_\alpha$ is countable, dense and closed under $f$. Now choose a new countable dense set (which is possible by $\aleph_1$-density) and close it under $f$. This will be $A_\beta$. \comment[id=PM]{I'm partial to taking an elementary chain $\seq{M_\alpha:\alpha<\w_1}$ of countable models with $f,A\in M_0$ and letting $A_\alpha:=M_\alpha\cap A$.}

    By assumption there are countably many $f_n:A\rightharpoonup A$, each of which is a partial increasing function so that $\bigcup_{n<\w} f_n= \bigcup_{\alpha<\w_1}[(A_{\alpha+1}\setminus A_\alpha)\times (A_{\alpha+1}\setminus A_\alpha)]$ and therefore in particular we can cover the graph of $f$ by these functions since each $A_{\alpha+1} \setminus A_\alpha$ was closed under $f$. Since $f$ is uncountable however that means some $n < \omega$ we have that $f_n \cap f$ is uncountable, which completes the proof. 
\end{proof}

We will need a separate combinatorial property of linear orders which will be useful in constructing slicewise coverable linear orders. 

\begin{definition} \label{good definition}
Let $A$ be an $\aleph_1$-dense set of reals. We say that $A$ is \emph{good} if there is an injective enumeration $A=\{r_\xi:\xi<\w_1\}$ so that the following holds: for every $n<\w$ and every $\seq{a_\alpha:\alpha<\w_1}$ with $a_\alpha\in [\w_1]^n$, there exist $\alpha<\beta$ such that $r_{a_\alpha(i)}\le r_{a_\beta(i)}$ for every $i<n$.
\end{definition}

In what follows, when we write ``$A=\{r_\xi:\xi<\omega_1\}$ is good", we mean that $\seq{r_\xi:\xi<\omega_1}$ is an (injective) enumeration witnessing the goodness of $A$.

The following is shown on \cite[p. 164]{Sh:106}.
\begin{lemma}
Let $A = \{r_\alpha : \alpha < \omega_1\}$ be a set of mutually generic Cohen reals over a fixed ground model $V$. Then $A$ is good in $V[A]$.
\end{lemma}

Being good can be preserved by finite support iterations. This was explained on \cite[p. 166]{Sh:106} but we provide a proof to make this presentation more self contained. First we need to define the class of posets preserving goodness.

\begin{definition}\label{ap}
Let $A=\{r_\xi:\xi<\w_1\}$ be good. A forcing poset $\P$ is said to be \emph{appropriate (for $A$)} iff for all $n\in\w$ and $\seq{(p_\alpha,a_\alpha):\alpha<\w_1}$ where $p_\alpha\in\P$ and $a_\alpha\in[\w_1]^n$, there exist $\alpha<\beta$ such that $p_\alpha\parallel p_\beta$ and $r_{a_\alpha(i)}\le r_{a_\beta(i)}$ for all $i<n$.
\end{definition}

Observe that being appropriate implies ccc by ignoring the $a_\alpha$ part. Also, by \cite[Lemma 12]{Sh:106}, if $A$ is good and $\P$ is appropriate then $\P$ forces that $A$ is still good. In fact this is an equivalent characterization.
\begin{lemma}\label{appropriate lemma}
Let $\P$ be a partial order and let $A$ be good. The following are equivalent.
\begin{enumerate}
\item $\P$ is appropriate for $A$.
\item $\P$ is ccc and preserves the goodness of $A$.
\end{enumerate}
\end{lemma}

\begin{proof}
(A)$\implies$ (B): See \cite[Lemma 12]{Sh:106}.
 
(B)$\implies$ (A): Fix a natural number $n < \omega$ and let $\{(p_\alpha, a_\alpha): \alpha \in \omega_1\}$ be a set of pairs with $p_\alpha \in \P$ and $a_\alpha \in [\omega_1]^n$ as in the definition of goodness. Let $\dot{Z}$ be a $\P$-name for the set of $\alpha<\omega_1$ so that $p_\alpha \in \dot{G}$. 

\begin{claim}
There is a $p \in \P$ which forces that $\dot{Z}$ is uncountable. 
\end{claim}

\begin{proof}[Proof of Claim]
Otherwise $\forces_\P$``$\dot{Z}$ is countable". By the ccc, there is some countable ordinal $\gamma$ so that $\forces_\P \dot{Z} \subseteq \check{\gamma}$. But this is not possible since for any $p_\alpha$ with $\alpha > \gamma$ we have that $p_\alpha \forces \check\alpha \in \dot{Z}$.
\end{proof}

Fix now a condition $p \in \P$ forcing that $\dot{Z}$ is uncountable. Let $G \ni p$ be generic over $V$ and work in $V[G]$. By assumption, $A$ is still good. Let $Z = \dot Z^G$. Consider now the set $\{(p_\alpha, a_\alpha): \alpha \in Z\}$. Since $A$ is good (in $V[G]$), there are $p_\alpha$ and $p_\beta$ in $G$ so that $r_{a_\alpha(i)}\le r_{a_\beta(i)}$ for all $i<n$. Also, since $p_\alpha$ and $p_\beta$ are both in $G$, they are compatible. But now, back in $V$ we have that $\alpha$ and $\beta$ witness compatibility. 
\end{proof}

We now sketch the aforementioned preservation result. 

\begin{lemma} \label{iteration}
    Finite support iterations of appropriate posets are appropriate.
\end{lemma}

\begin{proof}
    Let $\langle \P_\alpha, \dot\Q_\alpha : \alpha < \delta\rangle$ be a finite support iteration of some ordinal length $\delta$ so that for each $\alpha < \delta$ we have that $\forces_\alpha$``$\dot\Q_\alpha$ is appropriate". We want to show by induction on $\delta$ that $\P_\delta$ is appropriate. Note that since appropriate forcing notions are ccc we know that $\P_\delta$ is ccc. 
    
    The successor stage follow from Lemma \ref{appropriate lemma} above since clearly being good is preserved by two-step iterations. Let us therefore assume that $\delta$ is a limit ordinal and for all $\alpha < \delta$ we have $\P_\alpha$ is appropriate. Let $\{(p_\alpha, a_\alpha):\alpha \in \omega_1\}$ be a sequence of pairs consisting of a condition $p_\alpha$ and a sequence $a_\alpha$ of countable ordinals of some fixed size $n < \omega$. By thinning out, we can find an uncountable set $X \in [\omega_1]^{\omega_1}$ so that the set $\{{\rm supp} (p_\alpha) : \alpha \in X\}$ forms a $\Delta$-system. Let $\gamma < \delta$ be an ordinal above the maximum of the root. Consider now the set $\{(p_\alpha \restr \gamma, a_\alpha)$. Since $\P_\gamma$ is assumed to be appropriate we can find $\alpha < \beta \in X$ so that $p_\alpha \restr \gamma\parallel p_\beta \restr\gamma$ and $r_{a_\alpha(i)}\le r_{a_\beta(i)}$ for all $i<n$. Now however it follows that $p_\alpha$ and $p_\beta$ are compatible which completes the proof of this case.
\end{proof}

\section{A Sufficient Condition}

In this section we prove a sufficient condition for the existence of a linear order satisfying the conclusion of Main Theorem \ref{maintheorem1}.

\begin{theorem}\label{con}
Suppose that there exists $A\subset\R$ such that the following hold:
\begin{enumerate}
\item $A$ is dense in $\R$ and $\aleph_1$-dense and;
\item $A$ is slicewise coverable.
\end{enumerate}
Let $\Q$ be an $\aleph_1$-preserving forcing of size $\aleph_1$. Then $\Q$ forces that $A$ and $A^*$ are far.
\end{theorem}

\begin{proof}
Fix $A$ as in the statement of the theorem and enumerate its elements as $\{a_\alpha : \alpha \in \omega_1\}$. Suppose towards a contradiction that $\Q$ is a forcing notion of size $\aleph_1$ forcing that $A$ and $A^*$ are near. Let $\seq{q_\xi:\xi<\omega_1}$ be an injective enumeration of $\Q$. Let $\dot{\pi}$ be a $\Q$-name so that $\forces_\Q$`` $\dot{\pi}$ is an order isomorphism between uncountable subsets of $A$ and $A^*$". Without loss of generality we can assume that the lower cone of every element of $\Q$ is uncountable.

Let $\dot{\pi}_*$ denote the name for the uncountable, partial function from $\omega_1$ to $\omega_1$ defined in the extension by $\dot{\pi}_* (\alpha) = \beta$ if and only if $\dot{\pi}(a_\alpha) = a_\beta$. For each $p \in \Q$ and $\alpha \in \omega_1$ we can choose a triple $(q_{p, \alpha}, \beta_{p, \alpha}, \gamma_{p, \alpha})$ so that the following are satisfied:
\begin{enumerate}
    \item $q_{p, \alpha} \leq p$ and, if $\xi < \omega_1$ is such that $q_{p, \alpha} = p_\xi$, then $\xi > \alpha$;
    \item $\beta_{p, \alpha}, \gamma_{p, \alpha} \in (\alpha, \omega_1)$ (that is they are countable ordinals above $\alpha$);
    \item $q_{p, \alpha} \forces \dot{\pi}_*(\check{\beta}_{p, \alpha}) = \check{\gamma}_{p, \alpha}$
\end{enumerate}

Note that these exist since $\dot{\pi}$ is assumed to be uncountable (so there are $\beta_{p, \alpha}, \gamma_{p, \alpha} > \alpha$) and every lower cone is uncountable so we can always strengthen to such a $q$ should the original index of a condition forcing the above statement be too small. 

By considering a chain of elementary submodels, we can find a club $E \subseteq \omega_1$ with the following properties.
\begin{enumerate}
    \item If $\delta \in E$ then for every $\eta,\xi < \delta$ we have that $\beta_{p_\xi, \eta}, \gamma_{p_\xi, \eta} < \delta$ and, if $q_{p_\xi,\eta} = p_\zeta$, then $\zeta < \delta$ as well.
    \item For every $\delta\in E$, if $\delta':=\min(E\setminus(\delta+1))$, then $\{a_i:\delta\le i <\delta'\}$ is dense in $A$. \label{filter}
    \item $0\in E$.
    %\item For every $\delta \in E$ the set of $\{a_i: i\in min(E) \setminus \delta\}$ is a countable dense set. \comment[id=PM]{Do you mean the set $\{a_i:\delta\le i <\delta'$, where $\delta':=\min(E\setminus (\delta+1))$?}
\end{enumerate}
Fix now such a club $E$ and enumerate its elements as $E = \{\delta_\alpha : \alpha < \omega_1\}$ in order. Observe that condition \ref{filter} ensures that $A_\alpha:=\{a_i:i<\delta_\alpha\}$ forms a filtration. Choose functions $f_n$ which are increasing and witness slicewise coverability of $A$ with respect to this filtration. %$A_\alpha = \{a_i: i\in min(E) \setminus \delta_\alpha\}$. 
%By the way the club was constructed, for each $p \in \Q$ and $\alpha < \omega_1$ we must have an $n < \omega$ so that $f_n(a_{\beta_{p, \alpha}}) = a_{\gamma_{p, \alpha}}$. 
\begin{claim}
For each $p\in\Q$ and each $\alpha<\omega_1$, there exists $n\in\omega$ such that $f_n(a_{\beta_{p,\alpha}})=a_{\gamma_{p,\alpha}}$.
\end{claim}

\begin{proof}
Fix $p\in\Q$ and $\alpha<\omega_1$. There is a unique $i<\omega_1$ such that $\delta_i\le \alpha<\delta_{i+1}$. By (A) in the definition of $E$, $\beta_{{p,\alpha}},\gamma_{p,\alpha}<\delta_{i+1}$. On the other hand, by (B) in the definition of $(q_{p, \alpha}, \beta_{p, \alpha}, \gamma_{p, \alpha})$, we have that $\beta_{p,\alpha} > \alpha\ge \delta_i$ and similarly for $\gamma_{p,\alpha}$. Therefore, $(a_{\beta_{p,\alpha}},a_{\gamma_{p,\alpha}})\in (A_{i+1}\setminus A_i) \times (A_{i+1}\setminus A_i)$, and so, by the definition of the $f_m$'s, there is some $n\in\omega$ with $f_n(a_{\beta_{p,\alpha}})=a_{\gamma_{p,\alpha}}$.
\end{proof}

Consider the set
\[
S_n:=\left\{q\in\Q: \exists\alpha<\omega_1\exists p\in\Q [q=q_{p,\alpha}\wedge f_n(a_{\beta_{p,\alpha}})=a_{\gamma_{p,\alpha}}]\right\}.
\]

%Let $S_n$ be the set of $q_{p, \alpha}$ so that $f_n(a_{\beta_{p, \alpha}}) = a_{\gamma_{p, \alpha}}$. 

\begin{claim}
    For every $n < \omega$, the set $S_n$ is an antichain.
\end{claim}

\begin{proof}[Proof of Claim]
   Fix $n < \omega$. If $q_{p_0, \alpha_0}, q_{p_1, \alpha_1} \in S_n$ then we have that $f_n(a_{\beta_{p_i, \alpha_i}}) = a_{\gamma_{p_i, \alpha_i}}$ for $i < 2$ and hence $\{(a_{\beta_{p_0, \alpha_0}}, a_{\gamma_{p_0, \alpha_0}}), (a_{\beta_{p_1, \alpha_1}}, a_{\gamma_{p_1, \alpha_1}}) \}$ is an increasing function so $q_{p_0, \alpha_0}$ and $q_{p_1, \alpha_1}$ cannot be in the same generic as $\dot\pi$ was forced to be order reversing. 
\end{proof}

For each $n < \omega$ extend $S_n$ to some maximal antichain $I_n$. Let $\dot{\xi}_n$ denote the unique countable ordinal so that $p_{\dot{\xi}_n} \in \dot{G} \cap \check{I}_n$. Let $\dot\xi$ be a name for $\sup_{n\in\omega}\dot{xi}_n$. Since $\omega_1$ is not collapsed, it follows that $\Vdash \dot{\xi}<\omega_1^V$. Pick $\epsilon < \omega_1$ and $\xi < \w_1$ so that $p = p_\epsilon \forces \dot{\xi} = \check{\xi}$. Let $\alpha \in E$ be larger than $\epsilon$ and $\xi$ and finally consider the triple $(q_{p, \alpha}, \beta_{p, \alpha}, \gamma_{p, \alpha})$. If $\zeta<\omega_1$ is such that $q_{p, \alpha} = p_\zeta$, then $\zeta > \alpha > \xi$ by (B) in the definition of $q_{p,\alpha}$. Moreover, there is an $n < \omega$ so that $f_n(a_{\beta_{p, \alpha}})= a_{\gamma_{p, \alpha}}$ and hence $q_{p, \alpha} \in S_n$. Thus, $q_{p, \alpha}$ forces that $\dot{\xi}_n = \zeta$, contradicting that $p$, which was a weaker condition, forced that $\dot{\xi}_n < \dot{\xi} < \zeta$. 
\end{proof}

\section{Constructing a Model of $\MA$ with a slicewise coverable linear order} \label{section: construction}
We now complete the proof of Main Theorem \ref{maintheorem1} by constructing a model of $\MA + \cc = \kappa$ in which there is a slicewise coverable, $\aleph_1$-dense linear order. By Theorem \ref{con} this suffices.

\begin{theorem} \label{construction}
    Assume $\GCH$ and let $\kappa$ be a regular cardinal greater than $\aleph_1$. There is a ccc forcing notion $\P$ so that in any generic extension by $\P$ there is an $\aleph_1$-dense $A \subseteq\mathbb R$ so that the following hold.
    \begin{enumerate}
\item $2^{\aleph_0} = \kappa$ and $\MA$ holds and;
\item $A$ is slicewise coverable.
\end{enumerate}
\end{theorem}

Towards proving this theorem we need a forcing notion which will allow us to cover the slices of a given filtration.

\begin{definition}\label{Q_A}
Let $A=\{r_\xi:\xi<\w_1\}$ be $\aleph_1$-dense and good. Given a filtration $\vec A$ of $A$, let $\Q_{\vec A}$ be the set of finite, order-preserving partial functions $q:A\rightharpoonup A$ such that
\[
\forall \xi,\eta<\w_1[q(r_\xi)=r_\eta\to \exists\alpha<\w_1(r_\xi,r_\eta\in A_{\alpha+1}\setminus A_\alpha)].
\]
Order $\Q_{\vec A}$ by (reverse) inclusion.
\end{definition}

\begin{lemma}
    Let $A$ be $\aleph_1$-dense and good, and let $\vec{A} = \langle A_\alpha: \alpha \in \omega_1\rangle$ be a filtration. The following hold.
\begin{enumerate}
    \item $\Q_{\vec{A}}$ is ccc.
    \item $\Q_{\vec{A}}$ generically adds a function $f:A \to A$ which is defined on an uncountable set, is increasing and has the property that for every $\alpha < \omega_1$ we have $a \in A_{\alpha+1}\setminus A_\alpha$ if and only if $f(a) \in A_{\alpha+1}\setminus A_\alpha$ for every $a \in {\rm dom}(p)$.
    \item $\Q_{\vec{A}}$ is appropriate.
\end{enumerate}
\end{lemma}

\begin{proof}
To establish (A), we argue by contradiction. We can then let $n$ be the smallest positive integer so that there exists an uncountable antichain all of whose elements have size $n$. Let $W$ be an uncountable antichain of $\Q_{\vec A}$ with $|q|=n$ for all $q\in W$.

\begin{claim}
We may assume, shrinking $W$ if necessary, that for all $p,q\in W$ with $p\neq q$, $\dom(p)\cap \dom(q)=\emptyset$ and $\ran(p)\cap\ran(q)=\emptyset$.
\end{claim}

\begin{proof}
If $n=1$, then the elements of $W$ are all singletons, and the conclusion is easy. Suppose therefore that $n>1$ and let $W'\in [W]^{\aleph_1}$ be such that $\{\dom(q):q\in W'\}$ forms a $\Delta$-system with root $R$. For each $r\in R$, $q(r)$ can take values in a countable set (by the definition of $\Q_{\vec{A}}$), so, by the finiteness of $R$, we may find $W''\in [W']^{\aleph_1}$ so that, if $p,q\in W'$, then $p\cup q$ is a function. Now note that $\{q\restr(\dom(q)\setminus R):q\in W''\}$ is a family of conditions of size $n-|R|$, which moreover forms an antichain. By the minimality of $n$, $|R|=0$, so the first half of the conclusion follows upon replacing $W$ by $W''$.

To get pairwise disjoint ranges, simply apply the previous paragraph to the family $\{q^{-1}:q\in W\}$.
\end{proof}

Set, for every $q\in W$,
\[
u_q:=\{\xi<\omega_1:r_\xi\in\dom(q)\}
\]
and
\[
v_q:=\{\xi<\omega_1:r_\xi\in\ran(q)\}.
\]
We identify $u_q$ with its increasing enumeration, so that, for example, $u_q(i)$ denotes the $i$th element (in the ordinal ordering) of $\dom(q)$ whenever $i<n$.

For each $q\in W$ and $i,j<n$ with $i\neq j$, fix rational numbers $d_{ij}^q$ and $e_{ij}^q$ so that the following hold:

\begin{itemize}
\item $d_{ij}^q$ is between $r_{u_q(i)}$ and $r_{u_q(j)}$;
\item $e_{ij}^q$ is between $r_{v_q(i)}$ and $r_{v_q(j)}$;
\end{itemize}
Also, fix for every $q\in W$ a permutation $\sigma_q$ in\footnote{$S_n$ is the symmetric group with $n$-elements} $S_n$ so that, for every $i<n$, $q(r_{u_q(i)})=r_{v(\sigma_q(i))}$.

Shrinking $W$ if necessary, we may assume that, for each $i\neq j$, the map $q\mapsto (d_{ij}^q,e_{ij}^q)$ is constant, say with value $(d_{ij},e_{ij})$. Similarly, we may assume that the map $q\mapsto \sigma_q$ is constant, say with value $\sigma$.

We now apply the goodness of $A$ to the family of $2n$-tuples $\seq{(u_q,v_q):q\in W}$ to find $p,q\in W$ with $p\neq q$ and
\[
\forall i<n\left[r_{u_p(i)}\le r_{u_q(i)} \wedge r_{v_p(i)}\le r_{v_q(i)}\right].
\]
We will show that $p$ and $q$ are compatible, which will contradict that $W$ forms an antichain. To see this, fix $x\in \dom(p)$ and $y\in \dom(q)$ with $x<y$ we need to see that $p\cup q$ is order-preserving on $\{x,y\}$. Fix $i,j<n$ so that $x=r_{u_p(i)}$ and $y_{u_q(j)}$.  We now distinguish two cases. Suppose first that $i\neq j$. Suppose that $r_{u_p(i)}<r_{u_p(j)}$, the other case being symmetric. We have that $x=r_{u_p(i)}<d_{ij}<r_{u_p(j)}$ by the choice of $d_{ij}$, and therefore $r_{u_q(i)}<d_{ij}<r_{u_q(j)}=y$ by the same reason. In particular, $x<y$. On the other hand, since $p$ is order-preserving and $r_{u_p(i)}<r_{u_p(j)}$, we infer that $p(r_{u_p(i)})<p(r_{u_p(j)})$, hence
\[
p(x)=r_{v_p(\sigma(i))}<e_{\sigma(i)\sigma(j)}<r_{v_p(\sigma(j))}
\]
and so
\[
r_{v_q(\sigma(i))}<e_{\sigma(i)\sigma(j)}<r_{v_q(\sigma(j))}=q(y).
\]
Putting everything together, $p(x)<q(y)$, as desired.

We are now left with the case $i=j$. We have that $p(x)=r_{v_p(\sigma(i))}$ and $q(y)=r_{v_q(\sigma(i))}$. By the choice of $p$ and $q$, $r_{u_p(i)}\le r_{u_q(i)}$ and $r_{v_p(\sigma(i))}\le r_{v_q(\sigma(i))}$, which immediately yields that $x<y$ and $p(x)<q(y)$.

%The main point is item (A) so we begin with this. Towards a contradiction suppose that $U \subseteq \mathbb Q_{\vec{A}}$ were an uncountably antichain. Clearly there is a minimal $n < \omega$ so that we can assume $|p| = n$ for every $p \in U$. Moreover we can assume that $U$ is a $\Delta$-system - and this means that the root must be empty since otherwise we would contradict the minimality of $n$. By thinning out further, we can assume that there is at most one $p \in U$ so that $a \in {\rm dom}(p) \cup {\rm range}(p)$ for any $a \in A$. Thinning out one more time we can find \footnote{in the sense that $\{I_i:i<n\}$ is pairwise disjoint and so is $\{J_i:i<n\}$}, rational intervals $I_i,J_i$ so that $a_i \in I_i$ and $p(a_i) \in J_i$ for every $i \leq n$ and every $p \in U$, where $\{a_i:i<n\}$ is the increasing enumeration of $p$. Applying the fact that $A$ is good, we conclude that there are two conditions $p, q \in U$ so that for every $i$ we have $a_i < b_i$ and $p(a_i) < q(b_i)$ for every $i \leq n$ where $a_i$ is the $i^{\rm th}$-element of the ordinal enumeration of the domain of $p$ and $b_i$ is the same for $q$. However this implies that $p \cup q$ is a condition and hence these are compatible as otherwise there would be an $i, j \leq n$ with $i \neq j$ so that $a_i < b_j$ and $p(a_i) > q(b_j)$ (or vice versa - the cases are symmetric) but this cannot happen by the choice of the rational intervals. 

Item (B) is clear from the way the poset is constructed. 

The proof of (C) is a straightforward modification of that of (A), so we point out the only difference and leave the details to the interested reader. Given $\{(p_\alpha,a_\alpha):\alpha<\omega_1\}$ with $|p_\alpha|=n$ and $|a_\alpha|=m$ for every $\alpha<\omega_1$, apply goodness to the sequence of $2n+m$ tuples $(u_{p_\alpha},v_{p_\alpha},a_\alpha)$ rather than just $(u_{p_\alpha},v_{p_\alpha})$. The argument is then as before, \textit{mutatis mutandis}.
\end{proof}

\begin{lemma}
    Let $A$ be good and $\vec{A}$ be a filtration. The finite-support product $\prod_{n < \omega}^\textsf{fin} \Q_{\vec{A}}$ is ccc, appropriate and generically adds a family $\{f_n\}_{n < \omega}$ of partial, strictly increasing functions $A \to A$ which respect the slices of the filtration and $\bigcup_{n < \omega} f_n = \bigcup_{\alpha < \omega_1} (A_{\alpha+1} \setminus A_\alpha) \times (A_{\alpha+1} \setminus A_\alpha)$ 
\end{lemma}

\begin{proof}
    Since $\Q_{\vec{A}}$ is ccc when $A$ is good it follows that $\Q_{\vec{A}}$ forces itself to remain ccc by the fact that it is appropriate. It follows therefore that the finite-support product is itself ccc and in fact appropriate since finite support iterations of ccc, appropriate forcing notions are themselves appropriate. By a density argument, the union of the generic functions cover $\bigcup_{\alpha<\omega_1}(A_{\alpha+1}\setminus A_\alpha)\times (A_{\alpha+1}\setminus A_\alpha)$, and we are done.
\end{proof}

Let us now finish the proof of Theorem \ref{construction}. This is similar to the proof of \cite[Theorem 2]{Sh:106}.

\begin{proof}[Proof of Theorem \ref{construction}]
    Assume $\GCH$ and first add $\aleph_1$-many Cohen reals. Call this set $A$. Again by \cite{Sh:106}, this is a good set. By a finite support iteration of length $\kappa$ we can, using some bookkeeping, force with every ccc, appropriate partial order of size ${<}\kappa$. By standard arguments we get that in the final model $2^{\aleph_0} = \kappa$ and Martin's axiom holds for appropriate posets. By Lemma \ref{iteration} $A$ is still good. Also, as previously remarked (and observed in \cite{Sh:106}), in fact there are no ccc inappropriate posets so full $\MA$ holds. Moreover, we can now apply the forcing notion $\prod^\textsf{fin}_{n < \omega} \Q_{\vec{A}}$ to any filtration of $\vec{A}$ to obtain a sequence of functions witnessing slicewise-coverability for that filtration. This completes the proof. 
\end{proof}

\section{Discussion and Open Questions
}

We conclude the paper with some final observations and questions for further research. We first note that proof of Theorem \ref{construction} actually shows the following.
\begin{theorem} \label{good linear orders are SWC under MA}
Let $A \subseteq \mathbb R$ be good and $\aleph_1$-dense. If $\MA_{\aleph_1}$ holds then $A$ is slicewise coverable and hence cannot be made isomorphic to its reverse by any forcing notion of size $\aleph_1$ which preserves $\aleph_1$.
\end{theorem}

In fact, it is easy to see that the above results have nothing to do with $\aleph_1$. Therefore we get the following, where the notions of ``good" and ``slicewise coverable" are generalized above $\aleph_1$ in the obvious way. Note that if $A$ is a $\kappa$-dense set of mutually generic Cohen reals it will good.

\begin{theorem}
Let $\kappa$ be a cardinal and let $A \subseteq \mathbb R$ be good and $\kappa$-dense. If $\MA_{\kappa}$ holds then $A$ is slicewise coverable.
\end{theorem}

As mentioned in the introduction, a corollary of Theorem \ref{good linear orders are SWC under MA} is an alternative proof of a related (though different) fact shown in \cite{MooreTod17}. This theorem can also be understood in the context of {\em rigidity of} $\aleph_1$-sized structures under $\MA$. Under $\MA$, often structures of size $\aleph_1$ retain some strong ``incompactness" property by forcing notions preserving $\aleph_1$ or at least ccc forcing. For example this is true for Aronszajn trees, all of which are special (and hence indestructibly Aronszajn) under $\MA$ but not in general. If $A$ is $\aleph_1$-dense, then for every countable and dense (in itself) subset $B\subset A$, we have that $B\cong B^*$ by Cantor's Theorem, but $A$ need not be isomorphic (or even near) to its reverse.

We note that the implication ``good $\to$ slicewise coverable" can fail in the absence of $\MA$:

\begin{lemma}
It is consistent for any cardinal $\kappa$ that $\kappa < 2^{\aleph_0}$ and there is a good, $\kappa$-dense $A \subseteq \mathbb R$ which is not slicewise coverable.
\end{lemma}

We remark that if $\kappa=\aleph_1$ then adding $\aleph_1$-many Cohen reals over a model of $\CH$ will witness the lemma as $\CH$ will hold again in $V[A]$ and hence by the main result of \cite{Baum73}, there will be a ccc forcing notion of size $\aleph_1$ which adds an isomorphism between $A$ and $A^*$.

\begin{proof}
Fix a cardinal $\kappa$ and let $A$ be a $\kappa$-dense set of mutually generic Cohen reals. Work in $V[A]$, where we know that $A$ is good. We note the following two facts. First, as is well known, see e.g. \cite[p. 473]{Bls10}, in $V[A]$ the set of Cohen reals is not meager. However observe that any slicewise coverable set $B \subseteq \mathbb R$ must in fact be meager. To see this, fix any filtration $\vec{B} = \{B_\alpha: \alpha \in \kappa\}$ and any family of countably many increasing functions $f_n:B \to B$ as in the definition of slicewise coverability. Note that every $x \in B$ there will be an $n < \omega$ so that $f_n(x) = x$. Now, note that for any given $n < \omega$ we have that the graph of $f_n$ is nowhere dense in $\mathbb R^2$. The result now follows.
\end{proof}

In fact there is another way as well to see this which gives more information.
\begin{lemma}
$\MA (\sigma$-linked$)$ is consistent with a good linear order which is not slicewise coverable.
\end{lemma}

\begin{proof}
As one can force $\MA(\sigma$-linked$)$ by a $\sigma$-linked forcing over a model of $\CH$, it suffices to show that $\sigma$-linked forcing cannot add a countable family of increasing functions covering a filtration that was not already covered by such a family of functions. Therefore let $A \subseteq\mathbb R$ be a set which is not slicewise coverable and let $\vec{A}$ be such a filtration with no family of functions. Suppose towards a contradiction that $\P$ is $\sigma$-linked and let $\{\dot f_n: n < \omega\}$ be forced to be a family of increasing functions covering $\vec{A}$ as in the definition of slicewise coverability. Let $\P = \bigcup_{n < \omega} \P_n$ be the partition into countably many linked pieces. If $n, m \in \omega$ then let $\hat{f}_{n, m}:A \to A$ be the partial, increasing function defined by $\hat{f}_{n, m} (x) = y$ if and only if some $p \in \P_m$ forces that $\dot{f}_n(\check{x}) = \check{y}$. It is easy to verify that by linked-ness each $\hat{f}_{n, m}$ is a monotone increasing function and collectively they will cover the square of the filtration. This is a contradiction however, which completes the proof.
\end{proof}

Note that the above implies that the forcing notions of the form $\Q_{\vec{A}}$ are not in general $\sigma$-linked. It is not clear what the relationship is between good, essentially increasing and slicewise coverable in $\ZFC$. 

\begin{question}
What is the $\ZFC$-provable behavior of the classes of good, essentially increasing and slicewise coverable $\aleph_1$-dense suborders of $\mathbb R$ in $\ZFC$?
\end{question}
It would be particularly interesting to know whether examples like this are at the heart of the failure of $\BA$ under $\MA$. One way to phrase this is the following.
\begin{question}
Does $\MA$ imply that every $\aleph_1$-dense linear order is either reversible or good?
\end{question}

As we mentioned in the introduction, the original motivation for this work was to determine whether $\PFA$ for $\aleph_1$-sized posets (or even $\MM$ for $\aleph_1$-sized posets) suffices to prove $\BA$. While the results presented above hint at a negative answer, we cannot yet rule this out completely.

\begin{question}
Do either $\PFA$ or $\MM$ for $\aleph_1$-sized posets prove $\BA$?
\end{question}
By what we have shown it would be enough to answer the following in the affirmative, though this is not clear.
\begin{question}
If $A \subseteq\mathbb R$ is a good, $\aleph_1$-dense linear order, is the property that $A$ is good preserved by countable support iterations of proper forcing notions?
\end{question}

%On a similar note we can ask the following:
%\begin{question}
%For a given $\aleph_1$-dense $A$ and $B$ is there necessarily a cardinal preserving $\P$ which adds an isomorphism from $A$ to its reverse?
%\end{question} This is false by Moore Todorcevic

Observe finally that by the proof $\PFA$ implies $\BA$ in \cite{BaumPFA}, there is always a proper poset $\P$ of size continuum adding an isomorphism from between two given $\aleph_1$-dense $A$ and $B$. By the results above there may not always be one of size $\aleph_1$. This suggests a cardinal characteristic which might give some interesting information.

\begin{question}
Suppose $A, B \subseteq \mathbb R$ are $\aleph_1$-dense. What can be said about the least size of a $\P$ which is proper and adds an isomorphism between them? Can it be strictly between $\aleph_1$ and $\cc$?
\end{question}

\bibliographystyle{plain}
\bibliography{shlhetal, extrarefs}
\end{document}